\documentclass[11pt]{article}

\usepackage{graphicx}
\usepackage{latexsym,amssymb}
\usepackage{amsthm}
\usepackage{indentfirst}
\usepackage{amsmath}
\usepackage{color}
\usepackage{xcolor}
\usepackage{hyperref}
\usepackage{ulem}
\usepackage{float}
\usepackage{textcomp}
\usepackage{mathrsfs}

\newtheorem{theoremalph}{Theorem}

\newtheorem{Theorem}{Theorem}[section]
\newtheorem*{Theorem A}{Theorem A}

\newtheorem{Proposition}[Theorem]{Proposition}

\newtheorem{Lemma}[Theorem]{Lemma}

\newtheorem{Remark}[Theorem]{Remark}
\newtheorem{Corollary}[Theorem]{Corollary}

\newtheorem*{Claim}{Claim}

\newtheorem*{Acknowledgements}{Acknowledgements}

\newcommand{\enabstractname}{Abstract}

\newcommand{\frabstractname}{R\'{e}sum\'{e}}
\newenvironment{enabstract}{%
	\par\small
	\noindent\mbox{}\hfill{\bfseries \enabstractname}\hfill\mbox{}\par
	\vskip 2.5ex}{\par\vskip 2.5ex}

\renewcommand\top{{\rm top}}

\begin{document}

\newpage
\title{Volume growth and  topological entropy of certain partially hyperbolic systems}
\author{Dawei Yang \and  Yuntao Zang\footnote{ Y. Zang is the corresponding author. Y. Zang would like to thank the support of Shanghai Key Laboratory of Pure Mathematics and Mathematical Practice and the project funded by China Postdoctoral Science Foundation (2020TQ0098). D. Yang was partially supported by NSFC 11671288, 11822109, 11790274.}}
\date{}
\maketitle
\begin{enabstract}
	Let $f$ be a $C^{1}$ diffeomorphism on a compact manifold $M$ admitting a partially hyperbolic splitting $TM=E^{s}\oplus_{\prec} E^{1}\oplus_{\prec} E^{2}\cdots \oplus_{\prec}E^{l}\oplus_{\prec} E^{u}$ where $E^{s}$ is uniformly contracting, $E^{u}$ is uniformly expanding and  $\dim E^{i}=1,\,1\leq i\leq l.$ We prove an entropy formula w.r.t. the volume growth rate of subspaces in the tangent bundle:
	
$$h_{\top}(f)=\lim_{n\to+\infty}\frac{1}{n}\log\int\max_{V\subset T_{x}M}|\det Df_{x}^{n}|_{V}|\,d x.$$ 

\end{enabstract}
%\begin{frabstract}
%Soit $f$ un diff\'{e}omorphisme de classe $C^{1}$ défini sur une vari\'{e}t\'{e} compacte. Supposons-le partiellement hyperbolique avec une d\'ecomposition domin\'ee de la forme $TM=E^{s}\oplus_{\prec} E^{1}\oplus_{\prec} E^{2}\cdots \oplus_{\prec}E^{l}\oplus_{\prec} E^{u}$ o\`u $E^s$ est uniform\'ement contract\'e, $E^u$ uniform\'ement dilat\'ee et $\dim E^{i}=1,\,1\leq i\leq l.$ Nous exprimons alors l'entropie en fonction du taux de croissance du volume dans le fibr\'{e} tangent: $$h_{\top}(f)=\lim_{n\to+\infty}\frac{1}{n}\log\int\max_{V\subset T_{x}M}|\det Df_{x}^{n}|_{V}|\,d x.$$
%\end{frabstract}
%\vskip 0.01cm \noindent{\bf Keywords}: entropy, volume growth, partially hyperbolic.

\section{Introduction}There are many invariants (e.g., various entropies, volume growth) that can measure the complexity of a dynamical system. The relationships between them were studied. One interesting inequality is an upper bound of the topological entropy:
\begin{equation}\label{Przytychi upper bound of entropy}
h_{\top}(f)\leq\limsup_{n\to+\infty}\frac{1}{n}\log\int||(Df^{n}_{x})^{\wedge}||\,d x
\end{equation}
where $f$ is a smooth diffeomorphism on a compact Riemannian manifold $M$ and $(Df^{n}_{x})^{\wedge}$ is the induced (by $Df^{n}_{x}$) map between exterior algebras of the tangent spaces $T_{x}M$ and $T_{f^{n}x}M$. Here $||\cdot||$ is the norm on operators, induced from the Riemannian metric.

\smallskip

This was obtained by Przytychi \cite{Prz80}  for $C^{1+\alpha}$ diffeomorphisms. Later on, it has been proved by Kozlovski \cite{Koz98} that for $C^\infty$ diffeomorphisms, it is in fact an equality. In this work, we would like to extend Przytychi-Kazlovski's results in a different setting: some $C^1$ partially hyperbolic diffeomorphisms.

\medskip

Hereafter, we always assume that $V$ denotes a linear space (or subspace). Our main result is the following: 

\begin{theoremalph}\label{entropy of partially hyperbolic systems with multi 1D centers}
	Let $f$ be a $C^{1}$ diffeomorphism on a compact manifold $M$. Assume there is a partially hyperbolic splitting $TM=E^{s}\oplus_{\prec} E^{1}\oplus_{\prec} E^{2}\cdots \oplus_{\prec}E^{l}\oplus_{\prec} E^{u}$ where $E^{s}$ is uniformly contracting, $E^{u}$ is uniformly expanding and $\dim E^{i}=1,\,1\leq i\leq l.$ Then $$h_{\top}(f)=\liminf_{n\to+\infty}\frac{1}{n}\log\int\max_{V\subset T_{x}M}|\det Df_{x}^{n}|_{V}|\,d x=\limsup_{n\to+\infty}\frac{1}{n}\log\int\max_{V\subset T_{x}M}|\det Df_{x}^{n}|_{V}|\,d x.$$ 
	%	where $$||(Df^{n}_{x})^{\wedge}||=\max_{V\subset T_{x}M}|\det Df_{x}^{n}|_{V}|.$$
\end{theoremalph}
Note that $\|(Df^{n}_{x})^{\wedge}\|$ has a geometrical explanation: $$\|(Df^{n}_{x})^{\wedge}\|=\max_{V\subset T_{x}M}|\det Df_{x}^{n}|_{V}|.$$ In this paper, we will use that latter notation as it provides more direct computation.

The dynamics of hyperbolic systems was understood very well. Beyond uniform hyperbolicity, partially hyperbolic diffeomorphisms as in Theorem~\ref{entropy of partially hyperbolic systems with multi 1D centers}  inherit some strong hyperbolicity. But the dynamics of these systems is not as clear as the hyperbolic case. The splitting type in Theorem~\ref{entropy of partially hyperbolic systems with multi 1D centers} has been shown to be abundant among diffeomorphisms away from homolinic tangencies by Crovisier-Sambarino-Yang \cite{CSY15}. For some other related work about the partial hyperbolicity with multi 1-D centers, one can see \cite{DFP12}, \cite{CYZ17}.
%{\color{blue}For a smooth diffeomorphism $f$, the relationship between volume growth and topological entropy was studied.}

There are several other works which establish the relationship between entropy and the growth rates of volumes . Here we give a partial list.
\begin{itemize}
\item  In $C^\infty$ setting, Yomdin \cite{Yom87} showed a formula between the topological entropy and other form of volume growth on sub-manifolds (in contrast with our volume growth which is on the tangent space).  

\item Also in $C^\infty$ setting, Burguet \cite{Bur19} has shown that for Lebesgue almost every point $x\in M$, there is some invariant probability measure in the limit set of $\{\frac{1}{n}\sum_{j=0}^{n-1}\delta_{f^{j}(x)}\}_{n\in\mathbb N}$ with entropy larger than (or equal to) the volume growth rate on the direction with positive upper Lyapunov exponents. He also gave some counterexample to show that it is not true if the system only has finite regularity.

\item Cogswell \cite{Cog00} showed in $C^{2}$ setting that one can use the volume growth rate of one single local unstable manifold to bound above the metric entropy. Recently, a preprint \cite{Zan20} extends Cogswell's result which bounds the metric entropy also by a mixture between volume growth rate (or many other invariants) and positive Lyapunov exponents.

\item In the $C^{1+\alpha}$ setting, by using Pesin theory, Newhouse \cite{New88} showed that the metric entropy is bounded above by the volume growth rate of sub-manifolds which are transverse to the stable manifolds.

\item In some $C^1$ partially hyperbolic setting (or dominated splitting), Saghin \cite{Sag14}, Guo-Liao-Sun-Yang \cite{GLS18} proved that the metric entropy can be bounded above by a mixture between the positive Lyapunov exponents and the volume growth of some sub-manifold. In \cite{CCE15} and \cite{CYZ18}, for Lebesgue almost every point $x\in M$, a lower bound of the metric entropy of the measures in the limit set of $\{\frac{1}{n}\sum_{j=0}^{n-1}\delta_{f^{j}(x)}\}_{n\in\mathbb N}$ is established w.r.t. the sum of the Lyapunov exponents on the stronger sub-bundle.

\end{itemize}

Most of the known results (e.g. \cite{Sag14}, \cite{GLS18}, \cite{Prz80})  are concentrated on establishing inequalities for  entropies (upper bound) when the system are not smooth enough. Now we can get a more precise relationship between the topological entropy and the volume growth in the $C^1$ partially hyperbolic setting. So we have to do more work on the lower bound of the topological entropy. %The main difficulty is to estimate at points that are not Lyapunov regular.

Our strategy is to estimate the volume growth on dynamical balls from above and below. Roughly speaking, we prove that for sufficiently many points $x\in M$ and sufficiently small number $\delta>0$, $$\int_{B(x,n,\delta)}\max_{V\subset T_{z}M}|\det Df_{z}^{n}|_{V}|\,d z\approx\text{sub-exponentially small}.$$ 

%
%In light of this, to estimate $\int\max_{V\subset T_{x}M}|\det Df_{x}^{n}|_{V}|\,d x$, it is more or less transferred to the question of counting the number of dynamical balls. Hence at this point, we can relate the volume growth to the entropies (topological or measure-theoretical).
%
Usually, the above estimate of the volume growth on dynamical balls is established only for some Lyapunov regular points. But the corresponding estimation on Lyapunov regular points is not sufficient (the Lyapunov regular points might have zero Lebesgue measure) to get an entropy formula (equality). Therefore, we have to get a uniform bound on the volume growth on the dynamical balls at all points in $M$ (see Lemma \ref{volume growth of dynamical ball}). Unfortunately, Lemma \ref{volume growth of dynamical ball} only gives estimation of the volume growth on the stronger direction (w.r.t. to a dominated splitting). We come up with a general result in Section \ref{The volume growth rate along subspaces} to verify that the maximal volume growth can be achieved by the stronger direction not only for Lyapunov regular points but also for all points in $M$. This eventually leads us to the entropy formula.

\section{Volume growth of dynamical balls: estimation on topological entropy}
Let $f$ be a $C^1$ diffeomorphism on a compact manifold $M$. We first introduce the concept of partially hyperbolic splitting.

%Given $x\in M$, we will study the volume growth rate of subspaces in $T_{x}M$. We define 

%$$\kappa_{n}(x,j)\triangleq\max_{\substack{V\subset T_{x}M\\\dim V=j}}|\det Df_{x}^{n}|_{V}|,\quad 1\leq j\leq \dim M,$$
%
%$$\kappa(x,j)\triangleq\limsup_{n\to +\infty}\frac{1}{n}\log \kappa_{n}(x,j),\quad 1\leq j\leq \dim M,$$
%
%$$\kappa(x)\triangleq\max\{\kappa(x,j)\,|\,1\leq j\leq \dim M\}.$$

%We note that by definition, $\kappa(x,j)=\kappa(f(x),j),\,x\in M$.
We say $f$ admits a \emph{dominated splitting} $TM=E\oplus_{\prec}F$ if $E,F$ are both $Df$-invariant and there are two constants $C>0$ and $\lambda\in (0,1)$ such that for any integer $k\in \mathbb{N}$, any $x\in M$ and any non-zero vectors $v_{E}\in E(x)$, $v_{F}\in F(x)$, we have $$\frac{||Df^{k}_{x}(v_{E})||}{||v_{E}||}\leq C\lambda^{k}\cdot\frac{||Df^{k}_{x}(v_{F})||}{||v_{F}||}.$$
\medskip

We say $f$ admits a \emph{partially hyperbolic splitting} $TM=E^{s}\oplus E^{1}\oplus E^{2}\cdots \oplus E^{l}\oplus E^{u}$, if 
\begin{itemize}
	\item For each $i=0,1,\cdots,l$, $TM=(E^{s}\oplus E^{1}\oplus\cdots \oplus E^{i})\oplus_\prec ( E^{i+1}\oplus\cdots \oplus E^{l}\oplus E^{u})$ is a dominated splitting.
	\item $E^{s}$ is uniformly contracting and $E^{u}$ is uniformly expanding: there are two constants $C>0$ and $\lambda\in (0,1)$ such that for any integer $k\in \mathbb{N}$, any $x\in M$ and any non-zero vectors $v^{s}\in E^{s}(x)$, $v^{u}\in E^{u}(x)$, we have $$\frac{||Df^{k}_{x}(v^{s})||}{||v^{s}||}\leq C\lambda^{k},\quad \frac{||Df^{-k}_{x}(v^{u})||}{||v^{u}||}\leq C\lambda^{k}.$$
\end{itemize}

\subsection{Upper bound of topological entropy}
In this subsection, our goal is to bound above the topological entropy by the volume growth rate.

\begin{Proposition}\label{topological entropy bounded by liminf}
	Let $f$ be a $C^{1}$ diffeomorphism on a compact manifold $M$. Assume there is a partially hyperbolic splitting $TM=E^{s}\oplus_{\prec} E^{1}\oplus_{\prec} E^{2}\cdots \oplus_{\prec}E^{l}\oplus_{\prec} E^{u}$ with $\dim E^{i}=1,\,1\leq i\leq l.$ Then $$h_{\top}(f)\leq\liminf_{n\to\infty}\frac{1}{n}\log\int\max_{V\subset T_{x}M}|\det Df_{x}^{n}|_{V}|\,d x.$$
\end{Proposition}

Let $h(f,\mu)$ denote the metric entropy of an invariant measure $\mu$. Proposition \ref{topological entropy bounded by liminf} above is a direct application of the following result.

\begin{Lemma}\label{metric entropy bounded by liminf}
	Let $f$ be a $C^{1}$ diffeomorphism on a compact manifold $M$. Assume there is a partially hyperbolic splitting $TM=E^{s}\oplus_{\prec} E^{1}\oplus_{\prec} E^{2}\cdots \oplus_{\prec}E^{l}\oplus_{\prec} E^{u}$ with $\dim E^{i}=1,\,1\leq i\leq l.$ Then for any ergodic measure $\mu$, $$h(f,\mu)\leq\liminf_{n\to\infty}\frac{1}{n}\log\int|\det Df_{x}^{n}|_{F^{k}}|\,d x$$  where $F^{k}=E^{k}\oplus E^{k+1}\oplus\cdots \oplus E^{l}\oplus E^{u}$ and $k$ is the smallest integer such that the Lyapunov exponent on $E^{k}$ is non-negative.
\end{Lemma}
Theorem 1.2 in \cite{GLS18} gives a similar upper bound where the volume growth was established with `$\limsup$'. Here we improve their result by using `$\liminf$' for partially hyperbolic systems with multi 1-D centers. Similar ideas can also be found in \cite{Koz98} for $C^{1+\alpha}$ ($\alpha>0$) systems (see Section 3.3 in \cite{Koz98}).

We postpone the proof of Lemma \ref{metric entropy bounded by liminf}. As a consequence of Lemma \ref{metric entropy bounded by liminf}, we first prove Proposition \ref{topological entropy bounded by liminf}.
\begin{proof}[Proof of Proposition \ref{topological entropy bounded by liminf}]
	First we note that for any $1\leq k\leq l$ and $x\in M$, $$|\det Df_{x}^{n}|_{F^{k}}|\leq \max_{V\subset T_{x}M}|\det Df_{x}^{n}|_{V}|.$$ 
	Recall the variational principal (see \cite[Theorem 8.6]{Wal82}): $$h_{\top}(f)=\sup_{\mu}h(f,\mu)$$ where the supreme is taken over all ergodic measures.
	Then the conclusion is a direct consequence of Lemma \ref{metric entropy bounded by liminf}.
\end{proof}

Now the rest of this subsection is devoted to the proof of Lemma \ref{metric entropy bounded by liminf}.

Let $B(x,\delta)$ be the ball in $M$ centered at $x$ with radius $\delta$. We define the \emph{dynamical ball} $$B(x,n,\delta)=\bigcap_{i=0}^{n-1}f^{-i}(B(f^{i}(x),\delta)).$$

Assume there is a dominated splitting $TM=E\oplus_{\prec}F$. Given $\alpha>0$, we define the cone along $F$ with width $\alpha$:

$$\mathcal{C}^{\alpha}_{F}=\{v_{E}+v_{F}\,|\,x\in M, v_{E}\in E_{x}, v_{F}\in F_{x}, ||v_{E}||\leq \alpha\cdot||v_{F}||\}.$$

A sub-manifold $\gamma$ with $\dim \gamma=\dim F$ is called \emph{tangent to the cone} $\mathcal{C}^{\alpha}_{F}$ if for any $x\in\gamma$, $T_{x}\gamma\subset \mathcal{C}^{\alpha}_{F}$.

We next give a key estimation on the lower bound of the volume growth of dynamical balls.

\begin{Lemma}\label{volume growth of dynamical ball lower bound general version}
	Let $f$ be a $C^{1}$ diffeomorphism on a compact manifold $M$. Assume there is a dominated splitting $TM=E\oplus_{\prec} E^{cu}$. Let $\mu$ be an ergodic measure such that the Lyapunov exponents on $E$ are negative and the Lyapunov exponents on $E^{cu}$ are non-negative. Then for any $\varepsilon,\delta>0$,
	there are a subset $\Omega$ with $\mu(\Omega)>\frac{1}{2}$ and an integer  $N_{0}$ such that for any $x\in\Omega$, 
	$$\int_{B(x,n,\delta)} |\det Df^{n}_{z}|_{E^{cu}}|\,dz\geq e^{-n\varepsilon}, \quad n\geq N_{0}.$$
\end{Lemma}

\begin{proof}
	Given a small number $\varepsilon>0$(much smaller than the gaps between the Lyapunov exponents of $\mu$), define $$\Omega_{\varepsilon}^{N}=\{x\in M\,|\,m(Df_{x}^{n}|_{E^{cu}})\geq e^{-\frac{n\varepsilon}{2}},\forall\, n\geq N\}$$ where $m(A)$ denotes the minimum norm of the linear map $A$, i.e., $m(A)=\inf_{v}\frac{||A(v)||}{||v||}$.
	
	We note that $\mu(\Omega_{\varepsilon}^{N})\to 1$ as $N\to +\infty$.

	%	\begin{Claim}
	%		For any $\delta\leq\delta_{0}$ and for $\mu$ almost every point $x\in\Omega_{\varepsilon}^{N}$,  $$\int_{\widetilde{B}^{F^{k}}(x,n,\delta)} |\det Df^{n}|_{F^{k}}|\circ\exp_{x}(t)\,dt\geq XXXX \delta e^{-n\varepsilon}.$$
	%	\end{Claim}
	%	\begin{proof}
	%		By ergodicity, the forward orbits of a recurrent point $x$ visits $\Omega_{\varepsilon}^{N}$ with frequency larger than XXXX. 
	
	For $\mu$ almost every point $x\in\Omega_{\varepsilon}^{N}$, we define two increasing sequence of integers $n_{0}<n_{1}<n_{2}<\cdots$ and $r_{0}\leq r_{1}\leq r_{2}\leq\cdots$ inductively: 
	\begin{itemize}
		\item Define $r_{0}=n_{0}=0$,
		\item Suppose that $r_{i}$ and $n_{i}$ have been defined. If $f^{n_{i}}(x)\in \Omega_{\varepsilon}^{N}$, we then define $n_{i+1}=n_{i}+N$ and $r_{i+1}=r_{i}$. If $f^{n_{i}}(x)\notin \Omega_{\varepsilon}^{N}$, we define $n_{i+1}=n_{i}+r$ and $r_{i+1}=r_{i}+r$ where $r$ is the smallest positive integer such that $f^{n_{i}+r}(x)\in\Omega_{\varepsilon}^{N}$.
	\end{itemize}
\begin{figure}[h]
	\centering
	\includegraphics[width=0.80\textwidth]{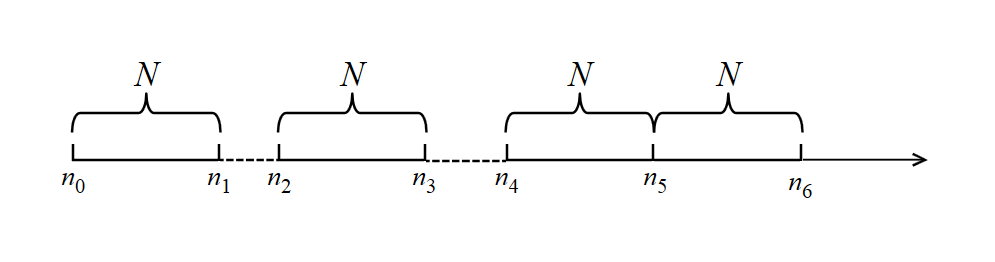}
	\caption{Definition of $n_{i}$}
\end{figure}

	Let $$j(n)=\max\{n_{i}\,|\,n_{i}\leq n\},\quad r_{x}(n)=N+n-j(n)+\max\{r_{i}\,|\,n_{i}\leq n\}.$$ 
	
	We note by definition, $$r_{x}(n)\leq N+\#\{i\leq n\,|\, f^{i}(x)\notin\Omega_{\varepsilon}^{N}\}.$$
	
	Let $\delta_{0}$ be a small number such that for any $z\in M$, the exponential map $\exp_{z}$ is a local diffeomorphism.

	For each $n\in\mathbb{N}$ and $\tau \leq\delta$, we use $\mathcal{W}(f^{n}(x),\tau)$ to denote the connected component in the ball $B(f^{n}(x),\tau)$ of $$ f^{n}\big(\exp_{x}(E^{cu}_{x}(\delta))\big)$$ which contains $f^{n}(x)$. Let $C=\max_{x\in M}||Df_{x}||$. We define $$D_{n}=f^{-n}(\mathcal{W}(f^{n}(x),C^{-r_{x}(n)}\cdot C^{-N}\delta \cdot e^{-(n-r_{x}(n))\varepsilon})).$$

By domination, there are some $T\in\mathbb{N}$ and some small number $\alpha>0$ (independent of $x$) such for any embedded sub-manifold $\gamma$ tangent to the cone $\mathcal{C}^{\alpha}_{E^{cu}}$ and any $k\geq T$, $f^{k}(\gamma)$ is also tangent to $\mathcal{C}^{\alpha}_{E^{cu}}$. By shrinking $\delta_{0}$ if necessary(independent of $x$), we can assume for any $x$ and its corresponding $D_{n}$, $f^{k}(D_{n})$ is tangent to the small cone $\mathcal{C}^{\alpha}_{E^{cu}}$ for $k\in\mathbb{N}$. Moreover we also assume that for any $\delta\leq\delta_{0}$, any $\gamma$ tangent to $\mathcal{C}^{\alpha}_{E^{cu}}$, any $k\in\mathbb{N}$ and any $y,z\in \gamma$, if for any $0\leq i\leq k-1$, $d(f^{i}(y),f^{i}(z)\leq\delta$, then
\begin{equation}\label{expansion error}
\frac{m(Df^{k}|_{T_{y}\gamma})}{m(Df^{k}_{z}|_{E^{cu}})}\geq e^{-\frac{k\varepsilon}{2}}.
\end{equation}
	\begin{Claim}
		For each $x\in\Omega_{\varepsilon}^{N}$ and each $n$, $D_{n}\subset B(x,n,\delta)\cap\exp_{x}(E^{cu}_{x}(\delta))$.
	\end{Claim}		
	\begin{proof}
		We first show that 	$D_{j}\subset D_{n_{i}}$ for each $i$ and each $j\geq n_{i}$. We prove it with two steps:
		\begin{enumerate}
			\item $D_{n_{i+1}}\subset D_{n_{i}}$ for each $i$,
			
			\item $D_{j}\subset D_{n_{i}}$ for each $n_{i}\leq j<n_{i+1}$.
			
		\end{enumerate}
		For the first property, it is sufficient to show that $f^{n_{i+1}-n_{i}}(f^{n_{i}}(D_{n_{i}}))$ covers $f^{n_{i+1}}(D_{n_{i+1}})$, i.e., $f^{n_{i+1}}(D_{n_{i+1}})\subset f^{n_{i+1}-n_{i}}(f^{n_{i}}(D_{n_{i}}))$. By definition of $r_{x}(n)$, the radius of $f^{n_{i}}(D_{n_{i}})$ is $$C^{-r_{x}(n_{i})}\cdot C^{-N}\delta \cdot e^{-(n_{i}-r_{i}-N)\varepsilon}$$ and the radius of $f^{n_{i+1}}(D_{n_{i+1}})$ is $$C^{-r_{x}(n_{i+1})}\cdot C^{-N}\delta \cdot e^{-(n_{i+1}-r_{i+1}-N)\varepsilon}.$$ If $f^{n_{i}}(x)\in \Omega_{\varepsilon}^{N}$, then by definition,
		
		\begin{itemize}
			\item $n_{i+1}-n_{i}=N, r_{i+1}-r_{i}=0$ and consequently, $r_{x}(n_{i})=r_{x}(n_{i+1})$,
			\item the expansion rate of $f^{n_{i+1}-n_{i}}$ (which is $f^{N}$) on $f^{n_{i}}(D_{n_{i}})$ is at least $e^{-N\varepsilon}$. Indeed, we note by definition that for any point $y$ in $f^{n_{i}}(D_{n_{i}})$ and any $0\leq j\leq N-1$, $d(f^{j}(f^{n_{i}}(x)),f^{j}(y))\leq\delta$. We then conclude by Equation (\ref{expansion error}) and the fact $f^{n_{i}}(x)\in \Omega_{\varepsilon}^{N}$.
		\end{itemize}
		As a consequence, the radius of $f^{n_{i+1}-n_{i}}(f^{n_{i}}(D_{n_{i}}))$ is at least larger than the radius of $f^{n_{i}}(D_{n_{i}})$ times the factor $e^{-N\varepsilon}$ (which is exactly the radius of $f^{n_{i+1}}(D_{n_{i+1}})$). By comparing the radii, we conclude that $f^{n_{i+1}-n_{i}}(f^{n_{i}}(D_{n_{i}}))$ covers $f^{n_{i+1}}(D_{n_{i+1}})$. 
		
		Now for the case that $f^{n_{i}}(x)\notin \Omega_{\varepsilon}^{N}$, by definition, we have 
		\begin{itemize}
			\item $n_{i+1}-n_{i}=r_{i+1}-r_{i}=r_{x}(n_{i+1})-r_{x}(n_{i})$,
			\item the contraction rate of $f^{n_{i+1}-n_{i}}$ on $f^{n_{i+1}}(D_{n_{i+1}})$ is at most $C^{-(r_{i+1}-r_{i})}$.
		\end{itemize}
		Again by comparing the radii, we then similarly conclude  that $f^{n_{i+1}-n_{i}}(f^{n_{i}}(D_{n_{i}}))$ covers $f^{n_{i+1}}(D_{n_{i+1}})$.
		
		To see the second property, for any $n_{i}\leq j<n_{i+1}$, by definition, we note that the difference between the two radii of $f^{n_{i}}(D_{n_{i}})$ and $f^{j}(D_{j})$ is a factor $C^{-(j-n_{i})}$ which is the maximal contraction for $f^{j-n_{i}}$. We then get that $f^{j-n_{i}}(f^{n_{i}}(D_{n_{i}}))$ covers $f^{j}(D_{j})$. This implies $D_{j}\subset D_{n_{i}}$.

		Now we have shown that $D_{j}\subset D_{n_{i}}$ for each $i$ and each $j\geq n_{i}$. As a direct consequence, it gives that for any $n$, any $y\in D_{n}$, any $n_{i}\leq n$ and any $n_{i}\leq j\leq \min\{n,n_{i}+r_{i+1}-r_{i}\}$, 
		\begin{equation}\label{bowen property}
		d(f^{j}(x),f^{j}(y))\leq C^{-N}\delta.
		\end{equation}
		To show $D_{n}\subset B(x,n,\delta)\cap\exp_{x}(E^{cu}_{x}(\delta))$, it is sufficient to verify that for any $y\in D_{n}$ and any $0\leq j\leq n$, 
		$d(f^{j}(x),f^{j}(y))\leq \delta$. By the inequality (\ref{bowen property}), it remains to check the property for each $n_{i}<j<n_{i+1}$ in the case of $n_{i+1}-n_{i}=N$ with $f^{n_{i}}(x)\in\Omega_{\varepsilon}^{N}$. Indeed, this is guaranteed by the fact that $d(f^{n_{i}}(x),f^{n_{i}}(y))\leq C^{-N}\delta.$ The proof of the claim is complete.

	\end{proof}
	Given a number $\delta>0$, we write $\gamma_{x}=B(x,n,\delta)\cap\exp_{x}(E^{cu}_{x}(\delta))$ for short.	Write $d=\dim M$.
	\begin{Claim}
		For any $0<\delta\leq\delta_{0}$, there are $N_{0},N_{1}\in\mathbb{N}$ and a subset $\widetilde{\Omega}_{\varepsilon}\subset \Omega_{\varepsilon}^{N_{1}}$ with $\mu(\widetilde{\Omega}_{\varepsilon})>\frac{1}{2}$ such that for every point $x\in\widetilde{\Omega}_{\varepsilon}$, $$\int_{\gamma_{x}} |\det Df^{n}|_{T_{t}\gamma_{x}}|\,dt\geq  e^{-3dn\varepsilon}, \quad n\geq N_{0}.$$
	\end{Claim}
	\begin{proof}
		%C^{(1-\mu(\Omega_{\varepsilon}^{N}))\cdot n\cdot\dim M}\cdot C^{-N}\delta \cdot e^{-n(1-\mu(\Omega_{\varepsilon}^{N})\varepsilon}.		
		For a recurrent point $x\in\Omega_{\varepsilon}^{N}$, by the previous claim, we have
		$$\begin{aligned}
		\int_{\gamma_{x}} |\det Df^{n}|_{T_{t}\gamma_{x}}|\,dt
		&\geq \int_{D_{n}} |\det Df^{n}|_{T_{t}\gamma_{x}}|\,dt\\
		&={\rm Vol}(\mathcal{W}(f^{n}(x),C^{-r_{x}(n)}\cdot C^{-N}\delta \cdot e^{-(n-r_{x}(n))\varepsilon}))\\
		&\geq \bigg(C^{-r_{x}(n)}\cdot C^{-N}\delta \cdot e^{-(n-r_{x}(n))\varepsilon}\bigg)^{\dim E^{cu}}.
		\end{aligned}$$	
		Recall that $r_{x}(n)\leq N+\#\{i\leq n\,|\, f^{i}(x)\notin\Omega_{\varepsilon}^{N}\}$ and $\mu(\Omega_{\varepsilon}^{N})\to 1$ as $N\to +\infty$. Also note that by recurrence (ergodicity), for $\mu$ almost every $x\in \Omega_{\varepsilon}^{N}$, $$\lim_{n\to+\infty}\frac{\#\{i\leq n\,|\, f^{i}(x)\in\Omega_{\varepsilon}^{N}\}}{n}= \mu(\Omega_{\varepsilon}^{N}).$$
		
		As a consequence, for $\mu$ almost every $x\in \Omega_{\varepsilon}^{N}$, $$\limsup_{n\to+\infty}\frac{r_{x}(n)}{n}\leq 1-\mu(\Omega_{\varepsilon}^{N}).$$
		
		Hence we can choose $ N_{1}$ large enough such that the set  $$\widetilde{\Omega}^{N_{1}}_{\varepsilon}=\{x\in \Omega_{\varepsilon}^{N_{1}}\,|\, \lim_{n\to +\infty}\bigg(C^{-r_{x}(n)}\cdot e^{-(n-r_{x}(n))\varepsilon}\bigg)^{\dim E^{cu}}\geq  e^{-2dn\varepsilon}\}$$ carries $\mu$-measure larger than $\frac{2}{3}$. We then choose $N_{0}$ large enough and a subset $\widetilde{\Omega}_{\varepsilon}\subset \widetilde{\Omega}_{\varepsilon}^{N_{1}}$ with $\mu(\widetilde{\Omega}_{\varepsilon})>\frac{1}{2}$ such that for any $x\in\widetilde{\Omega}_{\varepsilon}$, $$\bigg(C^{-r_{x}(n)}\cdot C^{-N_{1}}\delta \cdot e^{-(n-r_{x}(n))\varepsilon}\bigg)^{\dim E^{cu}}\geq e^{-3dn\varepsilon},\quad\forall\, n\geq N_{0}.$$ We complete the proof.
		
	\end{proof}

	Recall that we assume the Lyapunov exponents of $\mu$ on $E$ are negative. For a $C^{1}$ diffeomorphism with dominated splitting, F. Abdenur, C. Bonatti and S. Crovisier gives a non-uniform version of stable manifold theorem (see Proposition 8.9 in \cite{ABC11}) which states that for $\mu$ almost every point $x$, there is a local stable manifold $W^{s}_{\rm loc}(x)$ whose dimension is $\dim E$. Let $W^{s}_{\tau}(x)$ denote the local stable manifold with radius $\tau>0$. We note that the radii of these local stable manifolds might not be uniformly bounded from below (like Pesin theory in $C^{1+\alpha}$ setting). We choose a number $0<\tau\ll \delta$ small enough and a subset $\Omega\subset \widetilde{\Omega}_{\varepsilon}$ with $\mu$-measure larger than $\frac{1}{2}$ such that for each $x\in\Omega$, 
	\begin{itemize}
		\item the local stable manifold $W^{s}_{\rm loc}(x)$ has radius at least $\tau$, i.e., $W^{s}_{\tau}(x)\subset W^{s}_{\rm loc}(x)$.
		\item For any $n\in\mathbb{N}$, $W^{s}_{\tau}(x)\subset B(x,n,\delta)$.
	\end{itemize}	
	
	For $x\in\Omega$, we now consider the local foliation  $$\mathcal{F}_{x}=\{\exp_{x}\bigg(\exp^{-1}_{x}(z)+E^{cu}_{x}(\delta)\bigg)\,|\, z\in W^{s}_{x}(\tau)\}.$$
	
	For simplicity, we write $\gamma_{x}(z)=B(x,n,\delta)\cap\exp_{x}\big(\exp^{-1}_{x}(z)+E^{cu}_{x}(\delta)\big)$. 
	We note that by shrinking $\delta_{0}$ if necessary, we can assume that for any $\delta\leq\delta_{0}$, any $z\in B(x,n,\delta)$ and any $t\in\gamma_{x}$,	
	\begin{itemize}
		\item there is a smooth diffeomorphism $h_{z}:\gamma_{x}\to \gamma_{x}(z)$ with $ |\det(Dh_{z})|\geq \frac{1}{2}$,
		\item since $\gamma_{x}\subset B(x,n,\delta)$, $$\frac{|\det Df^{n}_{z}|_{E^{cu}}|}{|\det Df^{n}_{t}|_{T_{t}}\gamma_{x}|}\geq e^{-n\varepsilon}.$$
	\end{itemize}

	Applying Fubini's theorem to the foliation $\mathcal{F}_{x}$, by the above claim, we get that for any $n\geq N_{0}$,	
	$$\begin{aligned}
	\int_{B(x,n,\delta)} |\det Df^{n}_{z}|_{E^{cu}}|\,dz
	&\geq\int_{ W^{s}_{\tau}(x)}\bigg(\int_{\gamma_{x}(s)} |\det Df^{n}_{t}|_{E^{cu}}|\,dt\bigg)\,ds\\
	&\geq\frac{1}{2} e^{-n\varepsilon}\int_{W^{s}_{\tau}(x)}\bigg(\int_{\gamma_{x}} |\det Df^{n}_{t}|_{T_{t}\gamma_{x}}|\,dt\bigg)ds\\
	%			&=\tau^{\dim M-d^{k}}{\rm Vol}(\mathcal{W}(f^{n}(x),C^{-r_{x}(n)}\cdot C^{-N}\delta \cdot e^{-(n-r_{x}(n))\varepsilon}))\\
	&\geq \frac{1}{2}\tau^{\dim E}\cdot e^{-4dn\varepsilon}.
	\end{aligned}$$	
	Replacing $\varepsilon$ by $\frac{\varepsilon}{4d}$ and increasing $N_{0}$ if necessary, we can assume that for any $x\in\Omega$, $$\int_{B(x,n,\delta)} |\det Df^{n}_{z}|_{E^{cu}}|dz\geq e^{-n\varepsilon}\,,\quad\forall\, n\geq N_{0}.$$
	
	We get the desired estimation for all $\delta$ small enough. Since here we are considering a lower bound, this is true for all $\delta>0$. We complete the proof.
\end{proof}
A direct application of Lemma \ref{volume growth of dynamical ball lower bound general version} on our situation is the following:
\begin{Corollary}\label{volume growth of dynamical ball lower bound multi centers}
	Let $f$ be a $C^{1}$ diffeomorphism on a compact manifold $M$. Assume there is a partially hyperbolic splitting $TM=E^{s}\oplus_{\prec} E^{1}\oplus_{\prec} E^{2}\cdots \oplus_{\prec}E^{l}\oplus_{\prec} E^{u}$ with $\dim E^{i}=1,\,1\leq i\leq l.$ Then for any $\varepsilon,\delta>0$ and any ergodic measure $\mu$,
	there are a subset $\Omega$ with $\mu(\Omega)>\frac{1}{2}$ and an integer  $N_{0}$ such that for any $x\in\Omega$, 
	$$\int_{B(x,n,\delta)} |\det Df^{n}_{z}|_{F^{k}}|\,dz\geq e^{-n\varepsilon}, \quad n\geq N_{0}$$  where $F^{k}=E^{k}\oplus E^{k+1}\oplus\cdots \oplus E^{l}\oplus E^{u}$ and $k$ is the smallest integer such that the Lyapunov exponent of $\mu$ on $E^{k}$ is non-negative(write $k=u$ if the Lyapunov exponent on $E^{l}$ is still negative).
\end{Corollary}

Now we prove Lemma \ref{metric entropy bounded by liminf}. We first recall a result of Katok (Theorem 1.1 in \cite{Kat80}).

Given $\lambda>0$, let $S_{\lambda}(n,\tau)$ be the minimum number of dynamical balls $\{B(x,n,\tau)\}$ whose union has $\mu$-measure larger than or equal to $\lambda$. Recall that a subset $S\subset M$ is called a \emph{$(n,\tau)$ spanning set} if $\{B(x,n,\tau)\}_{x\in S}$ covers $M$. Moreover, the spanning set $S$ is called \emph{minimal} if its cardinality is smaller than or equal to the cardinality of any other spanning set. Given a subset $\Omega$, let $S(n,\tau,\Omega)$ denote a minimal $(n,\tau)$ spanning set of $\Omega$. 
\begin{Lemma}[Katok, \cite{Kat80}]\label{Katok}
	Let $f$ be a homeomorphism on a compact metric space $X$ and let $\mu$ be an ergodic measure.  Then for any $\lambda\in(0,1)$, $$h(f,\mu)=\lim_{\tau\to 0}\liminf_{n\to +\infty}\frac{1}{n}\log\# S_{\lambda}(n,\tau)=\lim_{\tau\to 0}\limsup_{n\to +\infty}\frac{1}{n}\log\# S_{\lambda}(n,\tau).$$
\end{Lemma}

\begin{proof}[Proof of Lemma \ref{metric entropy bounded by liminf}.]
Given $\varepsilon,\delta>0$ and an ergodic measure $\mu$, let $\Omega$ be the set with $\mu(\Omega)>\frac{1}{2}$ in Corollary \ref{volume growth of dynamical ball lower bound multi centers} and let $N_{0}$ be the corresponding integer. Let $S(n,2\delta,\Omega)$ be a minimal $(n,2\delta)$ spanning set of $\Omega$. By definition, for any $y_{1}, y_{2}\in S(n,2\delta,\Omega)$, $$B(y_{1}, n,\delta)\cap B(y_{2}, n,\delta)=\emptyset.$$

For any $n\geq N_{0}$, by Corollary \ref{volume growth of dynamical ball lower bound multi centers},

$$\begin{aligned}
\int_{M}|\det Df_{x}^{n}|_{F^{k}}|\,d x
&\geq \# S(n,2\delta,\Omega)\cdot\min_{x\in S(n,2\delta,\Omega)}\int_{B(x,n,\delta)} |\det Df_{y}^{n}|_{F^{k}}|\,dy\\
&\geq \# S(n,2\delta,\Omega)\cdot e^{-n\varepsilon}.
\end{aligned}$$
Then $$\liminf_{n\to\infty}\frac{1}{n}\log\int_{M}|\det Df_{x}^{n}|_{F^{k}}|\,d x+\varepsilon\geq \liminf_{n\to\infty}\frac{1}{n}\log S(n,2\delta
,\Omega).$$		

Since $\mu(\Omega)>\frac{1}{2}$, we have $S(n,2\delta,
\Omega)\geq S_{\frac{1}{2}}(n,2\delta)$. By Lemma \ref{Katok} and the arbitrariness of $\varepsilon$ and $\delta$, we get the result.

\end{proof}

\subsection{Lower bound of topological entropy}
In this subsection, we bound from below the topological entropy by the volume growth rate along the dominated sub-bundles. 

\begin{Proposition}\label{lower bound of topo entropy}
	Let $f$ be a $C^{1}$ diffeomorphism on a compact manifold $M$. Assume there is a partially hyperbolic splitting $TM=E^{s}\oplus_{\prec} E^{1}\oplus_{\prec} E^{2}\cdots \oplus_{\prec}E^{l}\oplus_{\prec} E^{u}$ with $\dim E^{i}=1,\,1\leq i\leq l.$ Then $$h_{\top}(f)\geq\limsup_{n\to\infty}\frac{1}{n}\log\int\max_{0\leq i\leq l}|\det Df_{x}^{n}|_{F^{i}}|\,d x$$ where the bundle $F^{i}\triangleq E^{i+1}\oplus E^{i+2}\oplus\cdots \oplus E^{l}\oplus E^{u}$.
\end{Proposition}
Combining Proposition \ref{lower bound of topo entropy} above with Proposition \ref{application to multi 1D centers} in the next section, we will get the desired lower bound of the topological entropy.

The key tool to prove Proposition \ref{lower bound of topo entropy} is to estimate from above the volume growth of the dynamical balls which is stated in the following. We remark that in Corollary \ref{volume growth of dynamical ball lower bound multi centers}, we estimate from below the volume growth of the dynamical balls for the partially hyperbolic systems with multi 1-D center. To estimate from above, we do not have to know on which sub-bundle the negative and non-negative Lyapunov exponents are separated. This brings two advantages: a general domination (not have to be partially hyperbolic) is enough and a probabilistic argument can be avoided so that we can consider the dynamical behavior of all points.
	
\begin{Lemma} \label{volume growth of dynamical ball}
	Let $f$ be a $C^1$ diffeomorphism on a compact manifold $M$. Assume there is a dominated splitting $TM=E\oplus_{\prec}F$. Then there is a constant $C$ such that for any $\varepsilon>0$, there is a constant $\delta>0$ such that for any $x\in M$ and $n\in\mathbb{N}$, 
	\begin{equation}\label{upper bound of volume growth on dynamical balls for E oplus F}
	\int_{B(x,n,\delta)} |\det Df_{y}^{n}|_{F}|\,dy\leq C e^{n\varepsilon}.		\end{equation}
\end{Lemma}

\begin{proof}

	%	For any $y\in B(x,\delta)$, let $\mathcal{F}_{x}^{\delta}(y)$ be the leaf in $\mathcal{F}_{x}^{\delta}$ that contains $y$.
	
	%	For any $\varepsilon>0$, we can choose $\delta$ small enough such that for any $x\in M$, $n\in\mathbb{N}$ and $y\in B(x,n,\delta)$, $$\frac{|\det Df^{n}_{y}|_{F}|}{|\det Df^{n}_{x}|_{F}|}\leq e^{n\varepsilon}.$$ 
	Indeed, to prove the lemma, it is sufficient to find some $\delta_{0}>0$ such that Equation (\ref{upper bound of volume growth on dynamical balls for E oplus F}) above holds for all $\delta\leq\delta_{0}$ with some constant $C$ only depending  on $\delta_{0}$.
	
	For $\delta>0$, write $$\widetilde{B}(x,n,\delta)\triangleq \exp_{x}^{-1}(B(x,n,\delta))\subset T_{x}M.$$
	
	First note that there is $\delta_{0}\ll 1$ small enough such that for any $\delta\leq\delta_{0}$, any $x\in M$ and any $n\in\mathbb{N}$, \begin{equation}\label{exp uniform control}
	\int_{B(x,n,\delta)} |\det Df_{y}^{n}|_{F}|\,dy\leq 2\int_{\widetilde{B}(x,n,\delta)} |\det Df^{n}|_{F}|\circ\exp_{x}(w)\,d w.
	\end{equation}

	%	\begin{itemize}
	%		\item  $$\frac{|\det Df^{n}_{y}|_{F}|}{|\det Df^{n}_{y}|_{T_{y}\mathcal{F}_{x}^{\delta}(y)}|}\leq e^{n\varepsilon},$$
	%		\item $$\frac{|\det Df^{n}_{y}|_{F}|}{|\det Df^{n}_{y}|_{T_{y}\mathcal{F}_{x}^{\delta}(y)}|}\leq e^{n\varepsilon},$$
	%	\end{itemize} 
	Applying Fubini Theorem to the foliation $\mathcal{F}_{x}\triangleq\{s+F_{x}\,|\,s\in E_{x}\},$ we get that 
	\begin{equation}\label{fubini theorem}
	\int_{\widetilde{B}(x,n,\delta)} |\det Df^{n}|_{F}|\circ\exp_{x}(w)\,dw=\int_{\widetilde{B}^{E}(x,n,\delta)}\big(\int_{\widetilde{B}^{F}_{s}(x,n,\delta)} |\det Df^{n}|_{F}|\circ\exp_{x}((t,s))\,dt\big)ds
	\end{equation}
	where $\widetilde{B}^{E}(x,n,\delta)=\widetilde{B}(x,n,\delta)\cap E_{x}$ and $\widetilde{B}^{F}_{s}(x,n,\delta)=\widetilde{B}(x,n,\delta)\cap (s+F_{x})$.
	
	Write for simplicity $$\gamma_{x}(s)\triangleq\exp_{x}(\widetilde{B}^{F}_{s}(x,n,\delta)).$$
	
	We fix a cone $\mathcal{C}^{\tau}_{F}$ for some number $\tau>0$ such that for each $\delta\leq\delta_{0}$, all the corresponding $\{\gamma_{x}(s)\}$ are tangent to $\mathcal{C}^{\tau}_{F}$. We also fix a constant $\hat{C}$ such that for any embedded sub-manifold $\gamma$ tangent to $\mathcal{C}^{\tau}_{F}$ with radius less than $\delta_{0}$, the induced Lebesgue volume is bounded above:	$$Leb(\gamma)\leq \hat{C}.$$ We note that the domination gives that there is some uniform integer $N\in\mathbb{N}$ such for any embedded sub-manifold $\gamma$ tangent to $\mathcal{C}^{\tau}_{F}$ and any $n\geq N$, $f^{n}(\gamma)$ is also tangent to $\mathcal{C}^{\tau}_{F}$.

	As a consequence, for any $\delta\leq\delta_{0}$ and any $x\in M$, 
	\begin{equation}\label{volume of disk tangent to the cone of F}
	Leb(f^{n}(\gamma_{x}(0)))\leq \hat{C},\quad n\geq N.
	\end{equation}

%	The domination gives that there is some $N\in\mathbb{N}$ such for any embedded sub-manifold $\gamma$ tangent to $\mathcal{C}^{\tau}_{F}$ and any $n\geq N$, $f^{n}(\gamma)$ is also tangent to $\mathcal{C}^{\tau}_{F}$. As a consequence, for any small number $\delta\leq\delta_{0}$ and any $x\in M$, the embedded submanifold $\gamma_{x}$ can be viewed as the graph of some continuous map from $E_{x}$ to $F_{x}$ with uniformly (w.r.t. $x$) bounded Lipschitz constant. Moreover, the domination 
%
%	%Since $TM=E\oplus_{\prec} F$ is a dominated splitting, the angle between $E$ and $F$ is uniformly bounded away from zero.	
%	
%	Therefore, we can choose a constant $\hat{C}$ (depending on $\delta_{0}$) such that for any $\delta\leq\delta_{0}$ and any $x\in M$, the induced Lebesgue volume 
%	\begin{equation}\label{volume of disk tangent to the cone of F}
%	Leb(\gamma_{x})\leq \hat{C}.
%	\end{equation} 

	Moreover, given any $\varepsilon>0$, by shrinking $\delta_{0}$ if necessary, we can assume that for any $\delta\leq\delta_{0}$, any $x\in M$, $n\in\mathbb{N}$, any $s\in\widetilde{B}^{E}(x,n,\delta)$ and $y\in \gamma_{x}(0), z\in \gamma_{x}(s)$, 
	\begin{equation}\label{small perturbation}
	\frac{|\det Df^{n}_{z}|_{F}|}{|\det Df^{n}_{y}|_{T_{y}}\gamma_{x}(0)|}\leq e^{n\varepsilon}.
	\end{equation}
%	Again by shrinking $\delta_{0}$ if necessary, we can also assume for $n\in\mathbb{N}$ and $s\in\widetilde{B}^{E}(x,n,\delta)$,
%	\begin{equation}
%	\int_{\widetilde{B}^{F}_{s}(x,n,\delta)} |\det Df^{n}|_{F}|\circ\exp_{x}(t,s)\,dt\leq 2 \int_{\gamma_{x}} |\det Df^{n}_{y}|_{F}| dy.
%	\end{equation}
	Also note that we can identity $\gamma_{x}(0)$ with any $\gamma_{x}(s)$ through a smooth diffeomorphism $h_{s}:\gamma_{x}(0)\to \gamma_{x}(s)$ with $ |\det(Dh_{s})|\leq 2$. Combining the formulas (\ref{volume of disk tangent to the cone of F}) and (\ref{small perturbation}), we have for $n\geq N$ and $s\in\widetilde{B}^{E}(x,n,\delta)$, $$\int_{\widetilde{B}_{s}^{F}(x,n,\delta)} |\det Df^{n}|_{F}|\circ\exp_{x}(t,s)\,dt\leq 2e^{n\varepsilon}\int_{\gamma_{x}(0)} |\det Df^{n}_{y}|_{T_{y}\gamma_{x}(0)}| dy= e^{n\varepsilon}Leb(f^{n}(\gamma_{x}(0)))\leq \hat{C} e^{n\varepsilon}.$$
	As a consequence, by letting $C_{0}=4\hat{C}$, (\ref{exp uniform control}) and (\ref{fubini theorem}) immediately give that $$\int_{B(x,n,\delta)} |\det Df_{y}^{n}|_{F}|\,dy\leq C_{0} e^{n\varepsilon},\quad n\geq N.$$	 Replacing $C_{0}$ by a larger number $C$ if necessary, we can assume that the above inequality holds w.r.t. $C$ for all $n\in \mathbb{N}$ and $\delta\leq\delta_{0}$.
\end{proof}

For partially hyperbolic systems with multi one dimensional centers, we have the following application of Lemma \ref{volume growth of dynamical ball}.
\begin{Corollary}\label{volume growth of dynamical ball multi centers}
	Let $f$ be a $C^{1}$ diffeomorphism on a compact manifold $M$. Assume there is a partially hyperbolic splitting $TM=E^{s}\oplus_{\prec} E^{1}\oplus_{\prec} E^{2}\cdots \oplus_{\prec}E^{l}\oplus_{\prec} E^{u}$ with $\dim E^{i}=1,\,1\leq i\leq l.$ Then there is a constant $C$ such that for any $\varepsilon>0$, there is $\delta>0$ such that for any $x\in M$ and any $n\in
	\mathbb{N}$, $$\int_{B(x,n,\delta)}\max_{0\leq i\leq l} |\det Df^{n}_{y}|_{F^{i}}| dy\leq C e^{n\varepsilon}$$ where the bundle $F^{i}=E^{i+1}\oplus E^{i+2}\oplus\cdots \oplus E^{l}\oplus E^{u}$.
\end{Corollary}
\begin{proof}
	By applying Lemma \ref{volume growth of dynamical ball} for $l$ times, we can get a constant $C$ such that for any $\varepsilon>0$, there is $\delta>0$ such that for any $x\in M$ and $n\in\mathbb{N}$, $$\max_{0\leq i\leq l}\int_{B(x,n,\delta)} |\det Df^{n}_{y}|_{F^{i}}| dy\leq C e^{\frac{n\varepsilon}{3}}.$$ On the other hand, by shrinking $\delta$, we can assume for any $p,q\in M$, any $0\leq i\leq l$ and any $n\in\mathbb{N}$, if $p\in B(q,n,\delta)$, then $$\frac{|\det Df^{n}_{q}|_{F^{i}}|}{|\det Df^{n}_{p}|_{F^{i}}|}\leq e^{\frac{n\varepsilon}{3}}.$$ Hence 
	$$\begin{aligned}
	\int_{B(x,n,\delta)} \max_{0\leq i\leq l}|\det 	Df^{n}_{y}|_{F^{i}}|\,dy
	&\leq e^{\frac{n\varepsilon}{3}}\int_{B(x,n,\delta)} \max_{0\leq i\leq l}|\det Df^{n}_{x}|_{F^{i}}| dy\\
	&= e^{\frac{n\varepsilon}{3}} \max_{0\leq i\leq l}\int_{B(x,n,\delta)}|\det Df^{n}_{x}|_{F^{i}}| dy\\
	&\leq e^{\frac{2n\varepsilon}{3}} \max_{0\leq i\leq l}\int_{B(x,n,\delta)}|\det Df^{n}_{y}|_{F^{i}}| dy\\
	&\leq C e^{n\varepsilon}.
	\end{aligned}$$	
\end{proof}

Now we are ready to give the lower bound of the topological entropy.

\begin{proof}[Proof of Proposition \ref{lower bound of topo entropy}]
	For $\delta>0$ and $n\in\mathbb{N}$, let $S(n,\delta, M)$ be a minimal spanning set of $M$. 
	%By definition, for any $y_{1}, y_{2}\in S(n,\varepsilon)$, $$B(y_{1}, n,\frac{\varepsilon}{2})\cap B(y_{2}, n,\frac{\varepsilon}{2})=\emptyset.$$
	By Corollary \ref{volume growth of dynamical ball multi centers}, there is a constant $C$ such that for any $\varepsilon>0$, there are some $\delta>0$ and  $n\in\mathbb{N}$,
	
	$$\begin{aligned}
	\int\max_{0\leq i\leq l}|\det Df_{x}^{n}|_{F^{i}}|\,d x
	&\leq \# S(n,\delta, M)\cdot\max_{x\in S(n,\delta, M)}\int_{B(x,n,\delta)} \max_{0\leq i\leq l}|\det Df_{y}^{n}|_{F^{i}}|\,dy\\
	&\leq \# S(n,\delta, M)\cdot C e^{n\varepsilon}.
	\end{aligned}$$
	Then $$\limsup_{n\to\infty}\frac{1}{n}\log\int\max_{0\leq i\leq l}|\det Df_{x}^{n}|_{F^{i}}|\,d x-\varepsilon\leq \limsup_{n\to\infty}\frac{1}{n}\log S(n,\delta, M
	).$$		
	
	Note that by the definition of topological entropy, for any $\delta>0$, the right side above is always less than or equal to the topological entropy. Hence by the arbitrariness of $\varepsilon$, we get the result.
\end{proof}
\begin{Remark}
	
	By a similar argument in Proposition \ref{lower bound of topo entropy}, for the case of dominated splitting $TM=E\oplus F$, Lemma \ref{volume growth of dynamical ball} gives an interesting result: $$h_{\top}(f)\geq\limsup_{n\to\infty}\frac{1}{n}\log\int|\det Df_{x}^{n}|_{F}|\,d x.$$
\end{Remark}

\section{The volume growth rate along subspaces}\label{The volume growth rate along subspaces}
Our goal in this section is to relate the volume growth rate along these subspaces ${F^{k}}$ in Proposition \ref{lower bound of topo entropy} to volume growth rate along all subspaces. 
\begin{Proposition}\label{application to multi 1D centers}
	Let $f$ be a $C^{1}$ diffeomorphism on a compact manifold $M$. Assume there is a partially hyperbolic splitting $TM=E^{s}\oplus_{\prec} E^{1}\oplus_{\prec} E^{2}\cdots \oplus_{\prec}E^{l}\oplus_{\prec} E^{u}$ with $\dim E^{i}=1,\,1\leq i\leq l.$ Then
	\begin{equation}\label{V equals to F i}
	\limsup_{n\to\infty}\frac{1}{n}\log\int\max_{V\subset T_{x}M}|\det Df_{x}^{n}|_{V}|\,d x= \limsup_{n\to\infty}\frac{1}{n}\log\int\max_{0\leq i\leq l}|\det Df_{x}^{n}|_{F^{i}}|\,d x
	\end{equation}
	where the bundle $F^{i}=E^{i+1}\oplus E^{i+2}\oplus\cdots \oplus E^{l}\oplus E^{u}$.
	%$$\max_{1\leq j\leq \dim M}\kappa(x,j)=\max_{1\leq i\leq l} \lambda(x, F^{i}(x))$$ 
\end{Proposition}
In order to prove the above result, we first show a general result that for a dominated splitting $TM=E\oplus_{\prec}F$, the maximal growth rate over subspaces with dimension $\dim F$ is uniformly bounded above by the growth rate of $F$.

\begin{Proposition}\label{relations between lambda and kappa}
	Let $f$ be a $C^1$ diffeomorphism on a compact manifold $M$. Assume there is a dominated splitting $TM=E\oplus_{\prec}F$. Then
	$$\limsup_{n\to+\infty}\frac{1}{n}\max_{x\in M} \bigg(\log\max_{\substack{V\subset T_{x}M\\\dim V=\dim F}}|\det Df_{x}^{n}|_{V}|- \log|\det Df_{x}^{n}|_{F}|\bigg)=0.$$
	%$$\lambda(x,F(x))=\kappa(x,\dim F).$$
\end{Proposition}
%\begin{Proposition}\label{relations between lambda and kappa}
%	Let $f$ be a $C^1$ diffeomorphism on a compact manifold $M$. Assume there is a dominated splitting $TM=E\oplus_{<}F$. Then for any $\varepsilon>0$, there is $N\in\mathbb{N}$ such that for any $x\in M$ and any $n\geq N$,
%	$$\frac{1}{n}\log \bigg(\max_{\substack{V\subset T_{x}M\\\dim V=\dim F}}|\det Df_{x}^{n}|_{V}|\bigg)-\frac{1}{n}\log |\det Df_{x}^{n}|_{F}|\leq \varepsilon.$$
%	%$$\lambda(x,F(x))=\kappa(x,\dim F).$$
%\end{Proposition}
%Note that even though vectors in $F$ direction have the relatively large expansion rate, the angles between different vectors may vary in a quite unpredictable way which would lead to a slow expansion rate of the determinant in $F$ direction.

We postpone the proof of Proposition \ref{relations between lambda and kappa}. We next use Proposition \ref{relations between lambda and kappa}  to prove Proposition \ref{application to multi 1D centers} in this section.

\begin{proof}[Proof of Proposition \ref{application to multi 1D centers}]
%	First note that it is sufficient to prove $$\limsup_{n\to\infty}\frac{1}{n}\log\int\max_{V\subset T_{x}M}|\det Df_{x}^{n}|_{V}|\,d x\leq\max_{0\leq i\leq l}\limsup_{n\to\infty}\frac{1}{n}\log\int|\det Df_{x}^{n}|_{F^{i}}|\,d x.$$
First we note that in formula (\ref{V equals to F i}), the left side is always larger than or equal to the right side. Hence we next show the reverse. 

For simplicity, we write $s=\dim E^{s}$ and $u=\dim E^{u}$. Given $\varepsilon>0$, applying Proposition \ref{relations between lambda and kappa} for $l$ times,  there is some $T_{1}\in\mathbb{N}$ such that for any $x\in M$ and any $n\geq T_{1}$, 
\begin{equation} \label{V equals to F i epsilon version}
\max_{0\leq i\leq l}\max_{\substack{V\subset T_{x}M\\\dim V=u+l-i}}|\det Df_{x}^{n}|_{V}|\leq e^{n\varepsilon}\cdot\max_{0\leq i\leq l}|\det Df_{x}^{n}|_{F^{i}}|.
\end{equation}	

\begin{Claim}
	There exists some $T_{2}\in\mathbb{N}$ such that for any $1\leq i_{s}\leq s$, $1\leq i_{u}\leq u$, any $x\in M$ and any $n\geq T_{2}$, 
	
	\begin{equation}\label{stable component}
	\max_{\substack{V\subset T_{x}M\\\dim V=i_{s}+l+u}}|\det Df_{x}^{n}|_{V}|\leq\max_{\substack{V\subset T_{x}M\\\dim V=l+u}}|\det Df_{x}^{n}|_{V}|,
	\end{equation}
	
	\begin{equation}\label{unstable component}
	\max_{\substack{V\subset T_{x}M\\\dim V=i_{u}}}|\det Df_{x}^{n}|_{V}|\leq\max_{\substack{V\subset T_{x}M\\\dim V=u}}|\det Df_{x}^{n}|_{V}|.
	\end{equation}
\end{Claim} 
We postpone the proof of the above claim. We use it to show that the right side is larger than or equal to the left side in formula (\ref{V equals to F i}). By considering the dimension of subspaces, formula (\ref{stable component}) and formula (\ref{unstable component}) give that for any $x\in M$ and any $n\geq T_{2}$, $$\max_{V\subset T_{x}M}|\det Df_{x}^{n}|_{V}|\leq \max_{0\leq i\leq l}\max_{\substack{V\subset T_{x}M\\\dim V=u+i}}|\det Df_{x}^{n}|_{V}|=\max_{0\leq i\leq l}\max_{\substack{V\subset T_{x}M\\\dim V=u+l-i}}|\det Df_{x}^{n}|_{V}|.$$

Together with formula (\ref{V equals to F i epsilon version}), we then get that for any $x\in M$ and any $n\geq \max\{T_{1},T_{2}\}$, $$\max_{V\subset T_{x}M}|\det Df_{x}^{n}|_{V}|\leq e^{n\varepsilon}\cdot\max_{0\leq i\leq l}|\det Df_{x}^{n}|_{F^{i}}|.$$

This implies $$	\limsup_{n\to\infty}\frac{1}{n}\log\int\max_{V\subset T_{x}M}|\det Df_{x}^{n}|_{V}|\,d x\leq \varepsilon+\limsup_{n\to\infty}\frac{1}{n}\log\int\max_{0\leq i\leq l}|\det Df_{x}^{n}|_{F^{i}}|\,d x.$$

By the arbitrariness of $\varepsilon$, we then get formula (\ref{V equals to F i}).
%Let $\alpha_{1},\alpha_{2},\cdots, \beta_{s}, \beta_{1},\cdots, \beta_{l}, \gamma_{1},\cdots, \gamma_{u}$ be a basis w.r.t. the dominated splitting $TM=E^{s}\oplus E^{1}\oplus E^{2}\cdots E^{l}\oplus E^{u}$. 
It remains to prove the claim.
\begin{proof}[Proof of Claim]

Let $T_{2}$ be a positive integer such that for any $x\in M$ and any $v\in E^{s}_{x}, w\in E^{u}_{x}$ with $||v||=1, ||w||=1$, $$||Df^{n}_{x}(v)||<1,\quad ||Df^{-n}_{x}(w)||<1,\quad \forall \,n\geq T_{2}.$$

Given $x\in M$, consider a subspace $V\subset T_{x}M$ with $\dim V=i_{s}+l+u$. Choose any $i_{s}$ dimensional subspace $V_{s}\subset E^{s}_{x}\cap V$. Let $V_{s}^{\perp}$ be the orthogonal subspace of $V_{s}$ in $V$. We note that $V_{s}^{\perp}$ has dimension $l+u$. Since $V_{s}\subset E^{s}_{x}$, for any $ n\geq T_{2}$, 

%$$\frac{1}{n}\log \bigg(|\det Df_{x}^{n}|_{V}|\bigg)\leq\frac{1}{n}\log \bigg(|\det Df_{x}^{n}|_{V_{s}^{\perp}}|\bigg)+ \frac{1}{n}\log \bigg(|\det Df_{x}^{n}|_{V_{s}}|\bigg).$$

	$$\begin{aligned}
|\det Df_{x}^{n}|_{V}|
	&\leq|\det Df_{x}^{n}|_{V_{s}^{\perp}}|\times|\det Df_{x}^{n}|_{V_{s}}|\\
	&\leq|\det Df_{x}^{n}|_{V_{s}^{\perp}}|.
	\end{aligned}$$

This proves inequality (\ref{stable component}). 

Next we consider a subspace $V\subset T_{x}M$ with $\dim V=i_{u}$. For any $n\geq T_{2}$, let $(Df^{n}_{x}(V))^{\perp}$ be the orthogonal subspace of $Df^{n}_{x}(V)$ in $T_{f^{n}x}M$. We note that $(Df^{n}_{x}(V))^{\perp}$ has dimension $\dim M-i_{u}$. Hence we can find a $u-i_{u}$ dimensional subspace $$V_{u}^{\perp}\subset Df^{-n}_{f^{n}x}\big((Df^{n}_{x}(V))^{\perp}\cap E^{u}_{f^{n}x}\big).$$

By definition, since $Df^{n}_{x}(V_{u}^{\perp})$ is orthogonal to $Df^{n}_{x}(V)$ and $Df^{n}_{x}(V_{u}^{\perp})\subset E^{u}_{f^{n}x}$, $V\oplus V_{u}^{\perp}$ has $u$ dimension and 

	$$\begin{aligned}
|\det Df_{x}^{n}|_{V\oplus V_{u}^{\perp}}|
	&=\bigg(|\det Df_{f^{n}x}^{-n}|_{Df^{n}_{x}(V)\oplus Df^{n}_{x}(V_{u}^{\perp})}|\bigg)^{-1}\\
	&\geq\bigg(|\det Df_{f^{n}x}^{-n}|_{Df^{n}_{x}(V)}|\times|\det Df_{f^{n}x}^{-n}|_{Df^{n}_{x}(V_{u}^{\perp})}|\bigg)^{-1}\\
	&\geq\bigg(|\det Df_{f^{n}x}^{-n}|_{Df^{n}_{x}(V)}|\bigg)^{-1}\\
	&= |\det Df_{x}^{n}|_{V}|.
	\end{aligned}$$

This proves the inequality (\ref{unstable component}).
\end{proof}

\end{proof}

Now it remains to prove Proposition \ref{relations between lambda and kappa}. We prepare some lemmas.

We next introduce the classical Oseledets' Theorem \cite{Ose68}. For more detailed discussions, see the book \cite{BaY07} and Appendix C.1 in \cite{BDV05}.

Let $f$ be a homeomorphism on a compact manifold $M$. Let $\pi:\mathscr{E}\to M$ be a finite-dimensional continuous vector bundle over $M$ endowed with the norm $||\cdot||$ induced by some inner product. Assume that $\Phi:\mathscr{E}\to\mathscr{E}$ is a continuous vector bundle automorphism. Write $\mathscr{E}_{x}$ the fiber of $\mathscr{E}$ at the point $x$ and $\Phi_{x}:\mathscr{E}_{x}\to\mathscr{E}_{f(x)}$ the action of $\Phi$ at the fiber $\mathscr{E}_{x}$. Define the linear cocycle induced by $\Phi$ as 
\begin{equation*}
\Phi^{n}_{x}\triangleq\left\{
\begin{array}{rcl}
\Phi_{f^{n-1}(x)}\circ \Phi_{f^{n-2}(x)}\circ\cdots\circ\Phi_{x}, & & {n\geq 1;}\\
\rm{Id}, & & {n=0;}\\
\Phi^{-1}_{f^{-n}(x)}\circ \Phi^{-1}_{f^{-n+1}(x)}\circ\cdots\circ\Phi^{-1}_{f^{-1}(x)}, & & {n\leq -1.}
\end{array} \right.
\end{equation*}

The Oseledets' Theorem \cite{Ose68} states that

\begin{Theorem}\label{Oseledets Theorem}
	Consider the system $(M, f, \mathscr{E},\Phi)$ as above. There are an $f$-invariant subset $R$ with total measure($\mu(R)=1$ for any invariant measure $\mu$), an $\Phi$-invariant measurable decomposition $\mathscr{E}=\mathscr{E}^{1}\oplus \mathscr{E}^{2}\oplus\cdots\oplus \mathscr{E}^{l}$ and finitely many measurable functions $\lambda_{1}(\cdot)<\lambda_{2}(\cdot)<\cdots<\lambda_{\rho(x)}(\cdot)$ such that for any $x\in R$ and any nonzero vector $v\in \mathscr{E}^{j}_{x}$, we have $$\lim_{n\to \pm\infty}\frac{1}{n}\log||\Phi^{n}_{x}(v)||=\lambda_{j}(x),$$
Moreover, for any $j\leq \rho(x)$,	$$\lim_{n\to \pm\infty}\frac{1}{n}\log|\det Df^{n}_{x}|_{\oplus_{i=1}^{j}\mathscr{E}^{i}}|=\sum_{i=1}^{j}\lambda_{i}(x) \cdot\dim \mathscr{E}^{i}.$$
\end{Theorem}

The set $R$ in Theorem \ref{Oseledets Theorem} is called the \emph{Lyapunov regular set} and $x\in R$ is called a Lyapunov regular point. These numbers $\lambda_{1}(\cdot)<\lambda_{2}(\cdot)<\cdots<\lambda_{\rho(x)}(\cdot)$ are called the \emph{Lyapunov exponents} of $x$ and the splitting $\mathscr{E}=\mathscr{E}^{1}\oplus \mathscr{E}^{2}\oplus\cdots\oplus \mathscr{E}^{l}$ is called the Osedelets' splitting. Usually the set $R$ is not the whole manifold $M$. 

%So points outside $R$ do not have Lyapunov exponents (the limits in Lemma \ref{Oseledets Theorem} do not exist). To conquer this, for any $x\in M$ and any nonzero vector $v\in T_{x}M$, we define the \emph{upper Lyapunov exponent} of $(x,v)$ as $$\lambda(x,v)\triangleq\limsup_{n\to +\infty}\frac{1}{n}\log ||\Phi^{n}_{x}(v)||.$$ 
%
%The following is some background of of the Lyapunov exponents and Oseledet flag. See the book \cite{BaY07} for more detail.
%\begin{Lemma}\label{Oseledet flag}
%	Consider the system $(M, f, \mathscr{E},\Phi)$. There are a measurable function $\rho:M\to\mathbb{N}$, a $\Phi$-invariant measurable filtration $0=V_{\rho(\cdot)+1}(\cdot)\subset V_{\rho(\cdot)}(\cdot)\subset V_{\rho(\cdot)-1}(\cdot)\subset\cdots\subset V_{1}(\cdot)=T_{(\cdot)}M$ and finitely many measurable functions numbers $\overline{\lambda}_{1}(\cdot)>\overline{\lambda}_{2}(\cdot)>\cdots>\overline{\lambda}_{\rho(\cdot)}(\cdot)$ such that for any $x\in M$, any $1\leq k\leq \rho(x)$ and any $v\in V_{k}\setminus V_{k+1}$, $\lambda(x,v)=\overline{\lambda}_{k}(x)$.
%\end{Lemma}

Let $f$ be a $C^1$ diffeomorphism on a compact manifold $M$. Now we apply the abstract Oseledets' Theorem (Theorem \ref{Oseledets Theorem}) to our situation with $\mathscr{E}$ being the tangent bundle and $\Phi$ being the derivative $Df$.

\begin{Lemma}\label{Oseledet flag for domination}
Let $f$ be a $C^1$ diffeomorphism on a compact manifold $M$. Assume there is a dominated splitting $TM=E\oplus_{\prec}F$. There is an $f$-invariant subset $R$ (the Lyapunov regular set) such that for any $x\in R$ and any $V_{x}\in T_{x}M$ with $\dim V_{x}=\dim F$,  $$\lambda(x,V_{x})\leq\lambda(x,F(x)).$$
%\begin{itemize}
%	\item There is $j$ such that $E(x)=V_{j}(x)$.
%	\item There is an $f$-invariant subset $R$ with total measure(the Lyapunov regular set) such that for any $x\in R$ and any $V_{x}\in T_{x}M$ with $\dim V_{x}=\dim F$,  $$\lambda(x,V_{x})\leq\lambda(x,F(x))=\sum_{i=1}^{j-1}\lambda_{i}(x)(\dim V_{i}-\dim V_{i+1}).$$
%\end{itemize}
\end{Lemma}
\begin{proof}
Write $d=\dim M$ and $k=\dim F$. Let $R$ be the Lyapunov regular set in Theorem \ref{Oseledets Theorem}. We write $\mathscr{E}^{1}\oplus \mathscr{E}^{2}\oplus\cdots\oplus \mathscr{E}^{l}$ the Oseledets' splitting (See Theorem \ref{Oseledets Theorem}). It is sufficient to prove that for any $x\in R$,
\begin{enumerate}
	\item there is a number $j_{F}\leq \rho(x)$ ($\rho(x)$ is from Theorem \ref{Oseledets Theorem}) such that $\dim F=\sum_{i=j_{F}}^{\rho(x)}\dim \mathscr{E}^{i}$. 
	\item $$\lambda(x,F(x))=\sum_{i=j_{F}}^{\rho(x)}\lambda_{i}(x)\dim \mathscr{E}^{i}.$$
	\item for any $V_{x}\subset T_{x}M$ with $\dim V_{x}=\dim F$, $$\lambda(x,V_{x})\leq \sum_{i=j_{F}}^{\rho(x)}\lambda_{i}(x)\dim \mathscr{E}^{i}.$$
\end{enumerate}

Next we prove the three properties above.
\paragraph{Property 1.} It is a consequence of domination by noting that there is some $j_{F}$ such that $$E(x)=\oplus_{i=1}^{j_{F}-1} \mathscr{E}^{i}_{x}.\eqno(*)$$

\paragraph{Property 2.}By domination, the angles between $E$ and $F$ is uniformly bounded below. As a consequence, $$\lim_{n\to +\infty}\frac{1}{n}\log|\det Df^{n}_{x}|\leq \lambda(x,E(x))+\lambda(x,F(x)).$$ For $x\in R$, by Theorem \ref{Oseledets Theorem}, we have $$\lim_{n\to +\infty}\frac{1}{n}\log|\det Df^{n}_{x}|=\sum_{i=1}^{\rho(x)}\lambda_{i}(x) \cdot\dim \mathscr{E}^{i}$$ and together with the formula $(\ast)$ above, we also have $$\lambda(x,E(x))=\sum_{i=1}^{j_{F}-1}\lambda_{i}(x) \cdot\dim \mathscr{E}^{i}.$$
We then get the second property.

\paragraph{Property 3.}Given $x\in R$, write $d_{i}=\dim \mathscr{E}^{i}_{x}$. Let $e_{1}, e_{2},\cdots,e_{d}$ be a basis of $T_{x}M$ such that  $e_{1},e_{2},\cdots, e_{d_{1}}\in \mathscr{E}^{1}_{x}$, $e_{d_{1}+1},e_{d_{1}+2},\cdots, e_{d_{1}+d_{2}}\in \mathscr{E}^{2}_{x}$, $e_{d_{1}+d_{2}+1},e_{d_{1}+d_{2}+2},\cdots, e_{d_{1}+d_{2}+d_{3}}\in \mathscr{E}^{3}_{x}$ and so on. Let $$\{v^{i}=\sum_{j=1}^{d}\delta^{i}_{j}\cdot e_{j}\,|\,\delta^{i}_{j}\in\mathbb{R}\}_{1\leq i\leq k}$$ be a basis of the subspace $V_{x}$. Write the coordinate matrix $\mathbb{A}$ of $\{v^{i}\}$ w.r.t. $\{e_{j}\}$ as $$
\mathbb{A}=[\mathbb{B}|\mathbb{C}]=\left[
\begin{array}{cccc|cccc}
\delta^{1}_{1} & \delta^{1}_{2} & \delta^{1}_{3} & \cdots             & \delta^{1}_{d-k+1}& \cdots         & \delta^{1}_{d-1} & \delta^{1}_{d}\\
\delta^{2}_{1} & \delta^{2}_{2} & \delta^{2}_{3} & \cdots             & \delta^{2}_{d-k+1}& \cdots         & \delta^{2}_{d-1} & \delta^{2}_{d}\\
\delta^{3}_{1} & \delta^{3}_{2} & \delta^{3}_{3} & \cdots             & \delta^{3}_{d-k+1}& \cdots         & \delta^{3}_{d-1} & \delta^{3}_{d}\\
\vdots         & \vdots         & \vdots         & \vdots             & \vdots            & \ddots         & \vdots           & \vdots\\
\delta^{k}_{1} & \delta^{k}_{2} & \delta^{k}_{3} & \cdots             & \delta^{k}_{d-k+1}& \cdots         & \delta^{k}_{d-1} & \delta^{k}_{d}
\end{array}
\right] 
$$ where $\mathbb{B}$ is a $k\times (d-k)$ matrix and $\mathbb{C}$ is a $k\times k$ matrix. Up to a change of the basis of $V_{x}$, we may assume the basis $\{v^{i}\}$ is such that the $k\times k$ matrix $\mathbb{C}$ above is lower triangular: 
 $$
 \mathbb{C}=\left[
 \begin{matrix}
\ast	& 0    & 0        & 0     & \cdots   & 0\\
\ast	& \ast & 0        & 0     & \cdots   & 0\\
\ast	& \ast & \ast     & 0	  & \cdots	 & 0\\
\ast    & \ast & \cdots   & \ast  & \cdots   & 0\\
\vdots  & \vdots & \ddots        & \vdots   &\vdots  & \vdots\\
\ast    & \ast & \cdots        &\ast        &\ast      & \ast   
 \end{matrix}
 \right] 
 .$$

%Let $$\{v^{i}=\sum_{j=1}^{d}\delta^{i}_{j}\cdot e_{j}\}_{1\leq i\leq k}$$ be a basis of the subspace $V_{x}$ such that for any $1\leq i\leq k-1$, $$\max\{1\leq j\leq d\,|\,\delta^{i}_{j}\neq 0\}>\max\{1\leq j\leq d\,|\,\delta^{i+1}_{j}\neq 0\}.$$ In another words, it means that for each $d-k+1\leq j\leq d$, the component $e_{j}$ only appears at most one time in the basis $\{v^{i}\}$ of $V_{x}$. Note that this special basis always exists by the diagonalization of a given basis. Indeed, we start from an arbitrary basis $\{w^{i}=\sum_{j=1}^{d}\alpha^{i}_{j}\}$ of $V_{x}$. By reordering the basis, we may assume  for any $1\leq i\leq k-1$, $$\max\{1\leq j\leq d\,|\,\alpha^{i}_{j}\neq 0\}\geq\max\{1\leq j\leq d\,|\,\alpha^{i+1}_{j}\neq 0\}.$$

Define $$\lambda(x,v^{i})\triangleq\limsup_{n\to+\infty}\frac{1}{n}\log||Df^{n}_{x}(v^{i})||.$$ The way we choose the basis (s.t. $\mathbb{C}$ is lower triangular) implies that $\lambda(x,v^{i})=\lambda(x,e_{t_{i}})$ where $$t_{i}=\max\{ j\,|\,\delta^{i}_{j}\neq 0\}.$$ Hence $$\lambda(x,V_{x})\leq \sum_{i=1}^{k}\lambda(x,v^{i})=\sum_{i=1}^{k}\lambda(x,e_{t_{i}}).$$  Since different $v^{i}$ corresponds to different $e_{t_{i}}$, $$\sum_{i=1}^{k}\lambda(x,e_{t_{i}})\leq \sum_{i=j_{F}}^{\rho(x)}\lambda_{i}(x)\dim \mathscr{E}^{i}.$$ We get the third property.

\end{proof}

The following abstract result is from Lemma 2 in \cite{Cao03}.
\begin{Lemma}\label{uniform upper bound for all points}
	Let $T$ be a continuous map on a compact metric space $X$ and let $\phi:X\to\mathbb{R}$ be a continuous function. Assume that for any invariant measure $\mu$, $\int\phi\,d\mu\leq 0$. Then $$\limsup_{n\to+\infty}\max_{x\in X}\frac{1}{n}\sum_{j=0}^{n-1} \phi(T^{j}(x))\leq 0.$$
\end{Lemma}

\begin{proof}[Proof of Proposition \ref{relations between lambda and kappa}] By Lemma \ref{Oseledet flag for domination}, there is an $f$-invariant subset $R$ with total measure such that for any $x\in R$ and any $V_{x}\subset T_{x}M$ with $\dim V_{x}=\dim F$, $$\lambda(x,V_{x})\leq\lambda(x,F(x)).$$ 

Write $k=\dim F$. Let $Gr(k,M)$ be the $k$-th Grassmannian manifold of $M$, i.e., $$Gr(k,M)=\{V\,|\,V\subset T_{x}M \text{\,is a linear space\,} \text{ with }\dim V=k, x\in M\}.$$ It is a compact metric space. Consider the induced map $\widetilde{f}:Gr(k,M)\to Gr(k,M)$ defined by $$\widetilde{f}(V_{x})=Df_{x}(V_{x}),\,V_{x}\in Gr(k,M).$$ 

Let $\mathscr{E}$ be the continuous vector bundle on $Gr(k,M)$ with $\mathscr{E}_{V_{x}}=T_{x}M, V_{x}\in Gr(k,M)$. Let $\Phi:\mathscr{E}\to\mathscr{E}$ be the bundle automorphism with $\Phi_{V_{x}}=Df_{x}, V_{x}\in Gr(k,M)$. Let $\pi$ be the projection from $Gr(k,M)$ to $M$ defined by $\pi(V_{x})=x$.  Let $\widetilde{R}$ be the Lyapunov regular set (see Lemma \ref{Oseledets Theorem}) w.r.t. the system $(Gr(k,M),\widetilde{f},\mathscr{E},\Phi)$. Note that by definition, $\pi(\widetilde{R})=R.$ Define $\phi:Gr(k,M)\to\mathbb{R}$ as $$\phi(V_{x})=\log|\det Df_{x}|_{V_{x}}|-\log|\det Df_{x}|_{F(x)}|.$$ By definition, if $V_{y}$ is close to $V_{x}$, then $y$ is close to $x$. As a consequence, $\log|\det Df_{y}|_{V_{y}}|$ is close to $\log|\det Df_{x}|_{V_{x}}|$ and $\log|\det Df_{y}|_{F(y)}|$ is close to $\log|\det Df_{x}|_{F(x)}|.$ Hence $\phi$ is a continuous function on $Gr(k,M)$. Since $\pi(\widetilde{R})=R$, given any $\widetilde{f}$-invariant measure $\widetilde{\mu}$ on $Gr(k,M)$, for $\widetilde{\mu}$ almost every point $V_{x}\in Gr(k,m)$, we have that $x\in R$. By Birkhoff ergodic theorem, $\frac{1}{n}\sum_{j=0}^{n-1}\phi(\widetilde{f}^{j}(V_{x}))$ converges. Then by the definition of $\phi$ and Lemma \ref{Oseledet flag for domination}, we have $$\lim_{n\to+\infty}\frac{1}{n}\sum_{j=0}^{n-1}\phi(\widetilde{f}^{j}(V_{x}))=\lambda(x,V_{x})-\lambda(x,F(x))\leq 0.$$ Hence $$\int\phi\,d\widetilde{\mu}=\lim_{n\to+\infty}\frac{1}{n}\sum_{j=0}^{n-1}\phi(\widetilde{f}^{j}(V_{x}))\leq 0.$$ Then by applying Lemma \ref{uniform upper bound for all points} to $\widetilde{f}$ and $Gr(k,M)$, we get that $$\limsup_{n\to+\infty}\max_{V_{x}\in Gr(k,m)}\frac{1}{n}\bigg(\log|\det Df^{n}_{x}|_{V_{x}}|-\log|\det Df^{n}_{x}|_{F}|\bigg)\leq 0.$$ This gives the  statement in Proposition \ref{relations between lambda and kappa}.

\end{proof}
\section{Proof of Theorem \ref{entropy of partially hyperbolic systems with multi 1D centers}}
With the preparations above, we now prove Theorem \ref{entropy of partially hyperbolic systems with multi 1D centers}.

The part $$h_{\top}(f)\leq\liminf_{n\to\infty}\frac{1}{n}\log\int\max_{V\subset T_{x}M}|\det Df_{x}^{n}|_{V}|\,d x.$$ is proved in Proposition \ref{topological entropy bounded by liminf}.

The other part is directly proved with the following two steps:

\begin{itemize}
	\item Proposition \ref{lower bound of topo entropy}: $$h_{\top}(f)\geq\limsup_{n\to\infty}\frac{1}{n}\log\int\max_{1\leq i\leq l}|\det Df_{x}^{n}|_{F^{i}}|\,d x$$ where the bundle $F^{i}=E^{i}\oplus E^{i+1}\oplus\cdots \oplus E^{l}\oplus E^{u}.$
	\item Proposition \ref{application to multi 1D centers}:  	$$\limsup_{n\to\infty}\frac{1}{n}\log\int\max_{V\subset T_{x}M}|\det Df_{x}^{n}|_{V}|\,d x= \limsup_{n\to\infty}\frac{1}{n}\log\int\max_{0\leq i\leq l}|\det Df_{x}^{n}|_{F^{i}}|\,d x.$$
\end{itemize}

The proof is now complete. \qed

\begin{Acknowledgements}
	We would like to thank J\'{e}r\^{o}me Buzzi, David Burguet, Jinhua Zhang and Jianyu Chen for helpful discussions.
\end{Acknowledgements}

\end{document}